\documentclass[a4paper,twoside,12pt]{article}

\usepackage[T1]{fontenc}
\usepackage{lmodern}
\usepackage{microtype}
\usepackage{amsmath,amssymb,amsthm,mathtools}
\usepackage{enumitem}
\usepackage{aliascnt}
\usepackage[top=2cm,bottom=2.5cm,left=2.5cm,right=2.5cm]{geometry}
\usepackage[colorlinks=true,linkcolor=blue,citecolor=blue,urlcolor=blue]{hyperref}
\usepackage[nameinlink,capitalize,noabbrev]{cleveref}

\hypersetup{
  pdftitle={Dense Ascending Waves and the Alon--Spencer Conjecture},
  pdfauthor={Yaping Mao},
  pdfsubject={The asymptotic order of ascending waves in half-dense subsets of the integers},
  pdfkeywords={ascending waves, dense subsets of integers, Ramsey theory, phase-space dynamics, extremal combinatorics}
}

\numberwithin{equation}{section}
\newcommand{\N}{\mathbb N}
\newcommand{\R}{\mathbb R}
\newcommand{\1}{\mathbf 1}
\newcommand{\cD}{\mathcal D}
\newcommand{\cE}{\mathcal E}
\newcommand{\cL}{\mathcal L}
\newcommand{\cR}{\mathcal R}
\newcommand{\cG}{\mathcal G}
\newcommand{\cX}{\mathcal X}
\newcommand{\lambdaTwo}{\lambda_2}

\theoremstyle{plain}
\newtheorem{theorem}{Theorem}[section]

\newaliascnt{lemma}{theorem}
\newtheorem{lemma}[lemma]{Lemma}
\aliascntresetthe{lemma}

\newaliascnt{conjecture}{theorem}
\newtheorem{conjecture}[conjecture]{Conjecture}
\aliascntresetthe{conjecture}

\newaliascnt{proposition}{theorem}
\newtheorem{proposition}[proposition]{Proposition}
\aliascntresetthe{proposition}

\newaliascnt{corollary}{theorem}
\newtheorem{corollary}[corollary]{Corollary}
\aliascntresetthe{corollary}

\theoremstyle{remark}
\newaliascnt{remark}{theorem}
\newtheorem{remark}[remark]{Remark}
\aliascntresetthe{remark}

\begin{document}

\title{\textbf{Dense ascending waves: A resolution of the Alon--Spencer conjecture}\thanks{Supported by the National Natural Science Foundation of China (Nos.~12471329 and 12061059).}}

\author{
Yaping Mao\thanks{Academy of Plateau Science and Sustainability, and School of Mathematics and Statistics, Qinghai Normal University, Xining, Qinghai 810008, China. E-mail: \texttt{yapingmao@outlook.com}; \texttt{myp@qhnu.edu.cn}}}
\date{}
\maketitle

\begin{abstract}
For a positive integer $n$, write $[n]=\{1,\ldots,n\}$. A strictly increasing sequence of integers $x_1<\cdots<x_k$ is an \emph{ascending wave} if its consecutive differences are nondecreasing. Let $g(n)$ be the largest integer $k$ such that every set $A\subseteq[n]$ with $|A|\ge n/2$ contains an ascending wave of length $k$. Alon and Spencer proved that
\[
 c_1\frac{(\log n)^2}{\log\log n}\le g(n)\le c_2(\log n)^2
\]
for all sufficiently large $n$, and they conjectured that the factor $\log\log n$ in the lower bound can be removed. In this paper, we confirm their conjecture.


\end{abstract}

\section{Introduction}\label{sec:introduction}

Throughout the paper, $\N=\{1,2,\ldots\}$, $[n]=\{1,\ldots,n\}$ for $n\in\N$, and $\log$ denotes the natural logarithm. If $S$ is a finite set, then $|S|$ denotes its cardinality; if $I\subseteq\R$ is an interval, then $|I|$ denotes its Euclidean length. The length of a finite sequence is the number of its terms.

A sequence of integers
$x_0<x_1<\cdots<x_m$
is an \emph{ascending wave} if
\[
 x_t-x_{t-1}\le x_{t+1}-x_t\qquad(1\le t<m).
\]
Thus the consecutive differences form a nondecreasing sequence. Sequences of one or two terms are ascending waves by convention, because there is no inequality to check. For example, $2,5,9,14$ is an ascending wave: its consecutive differences are $3,4,5$.

For $n\in\N$, let $g(n)$ be the largest integer $k$ with the following property: every set $A\subseteq[n]$ satisfying $|A|\ge n/2$ contains an ascending wave of length $k$. In this terminology, $A$ has density at least one half in $[n]$. We use the standard notation $f(n)=\Theta(h(n))$ to mean that there are constants $c,C>0$ and $n_0$ such that $c h(n)\le f(n)\le C h(n)$ for every $n\ge n_0$.

The study of wave configurations originates in work of Brown, Erd\H{o}s and Freedman~\cite{BEF}. Alon and Spencer~\cite{AS} subsequently proved that, for absolute constants $c_1,c_2>0$,
\[
 c_1\frac{(\log n)^2}{\log\log n}\le g(n)\le c_2(\log n)^2
\]
for all sufficiently large $n$. They conjectured that the logarithmic loss in the lower bound is not intrinsic. 

\begin{conjecture}[\cite{AS}]\label{ascon}
$g(n)=\Theta((\log n)^2)$.  
\end{conjecture}

Related variants impose strict inequalities between consecutive differences, zero-sum conditions, prescribed permutation patterns, or multicolour constraints; see~\cite{BBCY,BERJ,Cong,LR,LSR,RCDX}.

Alon and Spencer also identified the main difficulty in a direct greedy analysis: projection errors created at early times can focus many tentative trajectories into a later gap, so hitting events at different times cannot simply be treated as independent~\cite{AS}. Our proof addresses this obstruction by retaining the projection error as a second state variable. The resulting lift is invertible and area preserving, and repeated active hits on a fixed gap become disjoint image slices in phase space. Thus the focusing phenomenon is controlled by deterministic packing rather than by an independence assumption.

The purpose of this paper is to prove the conjectured lower bound.

\begin{theorem}\label{thm:main}
There is an absolute constant $c>0$ such that, for every $n\in\N$, every set $A\subseteq[n]$ with $|A|\ge n/2$ contains an ascending wave of length at least $c(\log n)^2$.
\end{theorem}

Together with the upper bound of Alon and Spencer~\cite{AS}, Theorem \ref{thm:main} implies $g(n)=\Theta((\log n)^2)$, proving Conjuecture \ref{ascon}.

In Section \ref{sec:dynamics} we introduce the projection and the phase-space dynamics. The local square-energy theorem is proved in Section \ref{sec:local}. The global scale selection and the multiscale splice are carried out in Section \ref{sec:global}, where Theorem \ref{thm:main} is completed.

\section{Projection and phase-space dynamics}\label{sec:dynamics}

Fix a nonempty finite set
 $A=\{a_1<a_2<\cdots<a_m\}\subseteq\mathbb Z$.
A \emph{gap of $A$} is an open interval
\[
 G=(a_r,a_{r+1})\qquad(1\le r<m),
\]
and its length is $|G|=a_{r+1}-a_r$. The word ``gap'' always refers to such an interval between two consecutive elements of $A$; the portions of the real line to the left of $a_1$ and to the right of $a_m$ are not called gaps.

For every real number $z\le a_m$, define the \emph{right projection of $z$ onto $A$} by
$
 R(z)=\min\{a\in A:a\ge z\}$,
and define the corresponding \emph{overshoot} by
 $\rho(z)=R(z)-z$.
Thus $R(z)=z$ and $\rho(z)=0$ when $z\in A$. If $z\le a_1$, then $R(z)=a_1$ and $\rho(z)=a_1-z$. If $z\in(a_r,a_{r+1})$, then $R(z)=a_{r+1}$ and $\rho(z)=a_{r+1}-z$. In particular,
$0<\rho(z)<|G|$ whenever $z\in G$.

The following lemma gives a condition under which the right projections form an ascending wave.

\begin{lemma}\label{lem:projection}
Let $u_0<u_1<\cdots<u_L$ be real numbers for which $R(u_t)$ is defined, and put $d_t=u_{t+1}-u_t$ for $0\le t<L$. Suppose that the projected points $R(u_0),\ldots,R(u_L)$ are strictly increasing and that
\[
 d_t-d_{t-1}=2\rho(u_t)\qquad(1\le t<L).
\]
Then $R(u_0),R(u_1),\ldots,R(u_L)$ is an ascending wave.
\end{lemma}

\begin{proof}
Write $x_t=R(u_t)=u_t+\rho(u_t)$. For $0\le t<L$,
\[
 x_{t+1}-x_t=d_t+\rho(u_{t+1})-\rho(u_t).
\]
Consequently, for $1\le t<L$,
\begin{align*}
 (x_{t+1}-x_t)-(x_t-x_{t-1})
 &=d_t-d_{t-1}+\rho(u_{t+1})-2\rho(u_t)+\rho(u_{t-1})\\
 &=\rho(u_{t+1})+\rho(u_{t-1})\ge0.
\end{align*}
Thus the consecutive differences of the projected sequence are nondecreasing.
\end{proof}

A point $p=(u,d)\in\R^2$ will be called a \emph{phase point}; $u$ is its position coordinate and $d$ is its difference coordinate. Set
\[
 \cD=\{(u,d)\in\R^2:u+d\le a_m\},
 \qquad
 \cX=(-\infty,a_m]\times\R.
\]
The \emph{shadow map} is the map $T:\cD\to\cX$ given by
 $T(u,d)=\bigl(u+d,\ d+2\rho(u+d)\bigr)$.
If $T^t p=(u_t,d_t)$ and $T^{t+1}p=(u_{t+1},d_{t+1})$, then
\begin{equation}\label{eq:shadow-recurrence}
 u_{t+1}=u_t+d_t,
 \qquad
 d_{t+1}=d_t+2\rho(u_{t+1}).
\end{equation}
In particular, $d_t-d_{t-1}=2\rho(u_t)$ for $t\ge1$.

As usual, $T^0$ is the identity map. For an integer $r\ge0$, we say that a phase point $p$ \emph{admits $r$ legal applications of $T$} if
\[
 T^jp\in\cD\qquad(0\le j<r).
\]
In this case $T^rp$ is defined; every phase point admits zero legal applications by convention. All subsets of $\R^2$ called measurable below are Lebesgue measurable, and $\lambdaTwo(B)$ denotes the two-dimensional Lebesgue measure, or area, of a measurable set $B\subseteq\R^2$. A Borel set is a member of the $\sigma$-algebra generated by the open subsets of $\R^2$.

We next record the basic invertibility, measure-preservation, and rollback
properties of the shadow map.
\begin{proposition}\label{prop:dynamics}
The shadow map $T$ is a Borel bijection from $\cD$ onto $\cX$, with inverse
$$
 T^{-1}(z,e)=\bigl(z-e+2\rho(z),\ e-2\rho(z)\bigr).
$$
If $B\subseteq\cD$ is measurable, then $T(B)$ is measurable and
$\lambdaTwo(T(B))=\lambdaTwo(B)$.
More generally, if every point of a measurable set $B$ admits $r$ legal applications of $T$, then $T^r(B)$ is measurable and
$\lambdaTwo(T^r(B))=\lambdaTwo(B)$.
Finally, if $s\ge t\ge0$ are integers, $p$ admits $t$ legal applications of $T$, $p'$ admits $s$ legal applications of $T$, and $T^t p=T^s p'$, then
\begin{equation}\label{eq:rollback}
 p=T^{s-t}p'.
\end{equation}
\end{proposition}

\begin{proof}
For $(u,d)\in\cD$, define
$S(u,d)=(u+d,d)$,
and, for $(z,e)\in\cX$, define
 $F(z,e)=(z,e+2\rho(z))$.
Since $(u,d)\in\cD$ implies $u+d\le a_m$, the composition $T=F\circ S$ is well defined and equals the stated shadow map.

The map $S$ is a linear bijection from $\cD$ onto $\cX$ with determinant $1$. Partition $(-\infty,a_m]$ into the disjoint half-open intervals
\[
 I_0=(-\infty,a_1],
 \qquad
 I_r=(a_r,a_{r+1}]\quad(1\le r<m).
\]
On each $I_r$, the function $\rho$ is affine. Hence the restriction of $F$ to $I_r\times\R$ is an invertible affine map with linear part
\[
 \begin{pmatrix}
 1&0\\
 -2&1
 \end{pmatrix},
\]
whose determinant is $1$.

Given $(z,e)\in\cX$, the equations
\[
 z=u+d,
 \qquad
 e=d+2\rho(z)
\]
give
\[
 d=e-2\rho(z),
 \qquad
 u=z-e+2\rho(z).
\]
Therefore $T$ is bijective and
$T^{-1}(z,e)=\bigl(z-e+2\rho(z),\,e-2\rho(z)\bigr)$.
The half-open partition makes the formula single-valued at the points of $A$. Since $\rho$ is Borel, both $T$ and $T^{-1}$ are Borel measurable.

We next justify the assertion for Lebesgue-measurable sets. An invertible affine map is bi-Lipschitz. It therefore maps Borel sets to Borel sets and Lebesgue-null sets to Lebesgue-null sets. Since every Lebesgue-measurable set differs from a Borel set by a null set, every invertible affine map sends Lebesgue-measurable sets to Lebesgue-measurable sets. Its effect on Lebesgue measure is multiplication by the absolute value of its determinant.

Let $B\subseteq\cD$ be measurable. The set $S(B)$ is measurable and has the same area as $B$. Split it into the finitely many measurable pieces
\[
 B_r=S(B)\cap(I_r\times\R)\qquad(0\le r<m).
\]
On each piece, $F$ is an invertible affine map of determinant $1$, so $F(B_r)$ is measurable and
\[
 \lambdaTwo(F(B_r))=\lambdaTwo(B_r).
\]
Because $F$ preserves the first coordinate and the intervals $I_r$ are disjoint, the sets $F(B_r)$ are pairwise disjoint. Hence
\[
 \lambdaTwo(T(B))
 =\sum_{r=0}^{m-1}\lambdaTwo(F(B_r))
 =\sum_{r=0}^{m-1}\lambdaTwo(B_r)
 =\lambdaTwo(B).
\]
Applying this argument successively proves the same statement for every finite legal iterate.

Finally, suppose that $T^tp=T^sp'$ with $s\ge t$. Applying the single-valued inverse $t$ times along the two legal orbit segments gives
$p=T^{s-t}p'$,
as claimed.
\end{proof}

\section{The local square-energy theorem}\label{sec:local}

We now fix the numerical constants used in the local argument:
\begin{equation}\label{eq:local-constants}
 C_{\mathrm{loc}}=20,
 \qquad
 \varepsilon=\frac1{40}.
\end{equation}
These values are chosen only for convenience and are not optimized.

For an interval $I\subseteq\R$ and a scale $D>0$, define the \emph{short-gap square energy of $I$ at scale $D$} by
\begin{equation}\label{eq:short-gap-energy}
 \cE_D^{<}(I)
 =\sum_{\substack{G\cap I\ne\emptyset\\ |G|<D}}
   \frac{|G|^2}{D},
\end{equation}
where the sum ranges over the gaps of $A$. Thus every short gap that intersects $I$ is counted, including a gap that crosses an endpoint of $I$.

Let $D\ge2$ be a real number, let $L\ge1$ be an integer, and let $y\in A$. Define the \emph{local window}
\[
 J=[y,y+C_{\mathrm{loc}}DL]=[y,y+20DL].
\]
We assume throughout this section that $J\subseteq(-\infty,a_m]$ and that every gap of $A$ intersecting $J$ has length less than $D$. Under this assumption, $\cE_D^{<}(J)$ is the sum of $|G|^2/D$ over all gaps meeting $J$.

The initial phase box is
$\cR_0=[y+D/2,y+D]\times[10D,11D].$
Its area is
\begin{equation}\label{eq:box-area}
 \lambdaTwo(\cR_0)=\frac{D^2}{2}.
\end{equation}
Moreover, $\cR_0\subseteq\cD$, because for $(u_0,d_0)\in\cR_0$,
\[
 u_0+d_0\le y+12D\le y+20DL\le a_m.
\]

For $t\ge0$, let $\cL_t\subseteq\cR_0$ be the set of initial phase points that admit $t$ legal applications of $T$. Thus $\cL_0=\cR_0$, and for $p\in\cL_t$ we write
$T^t p=(u_t(p),d_t(p))$.

For $p\in\cL_t$, define the accumulated overshoot
\begin{equation*}
S_t(p)=\sum_{r=1}^t\rho(u_r(p)),
 \qquad S_0(p)=0.
\end{equation*}
The attempted step from time $t-1$ to time $t$ is called \emph{active} when $p\in\cL_{t-1}$ and $S_{t-1}(p)<D$. For a fixed gap $G$ and $1\le t\le L$, the \emph{active hitting set}
\begin{equation*}
 H_t(G)=\{p\in\cL_{t-1}:S_{t-1}(p)<D,\ u_{t-1}(p)+d_{t-1}(p)\in G\}
\end{equation*}
consists of initial conditions in the common phase box $\cR_0$ whose active $t$th tentative position falls in $G$.

The core of the local argument is the following packing estimate.
\begin{lemma}\label{lem:single-gap}
Under the preceding assumptions, every gap $G$ intersecting $J$ satisfies
\begin{equation}\label{eq:single-gap-packing}
\sum_{t=1}^L\lambdaTwo(H_t(G))\le5D|G|.
\end{equation}
\end{lemma}

\begin{proof}
We first record the active-window estimate. Suppose that $p\in\cL_{t-1}$ and $S_{t-1}(p)<D$, where $1\le t\le L$. From \eqref{eq:shadow-recurrence},
\begin{equation}\label{eq:d-sum}
 d_r=d_0+2S_r\qquad(0\le r<t).
\end{equation}
Since $d_0\le11D$ and $S_r\le S_{t-1}<D$, we have $d_r<13D$ for $0\le r<t$. Therefore the tentative position at time $t$ satisfies
\begin{align*}
 u_{t-1}+d_{t-1}
 &=u_0+\sum_{r=0}^{t-1}d_r\\
 &<y+D+13tD\\
 &\le y+14LD<y+20LD.
\end{align*}
It is also larger than $y$. Hence it lies in $J$, so it is at most $a_m$ and the $t$th iterate is legal. Denoting this new position by $u_t$, the assumption on gaps meeting $J$ gives $\rho(u_t)<D$. Consequently,
\begin{equation}\label{eq:active-d-bound}
 10D\le d_t=d_0+2S_t<15D.
\end{equation}
This proves, in particular, that an orbit cannot become illegal while its accumulated overshoot is still below $D$.

The legal sets $\cL_t$, the functions $u_t,d_t,S_t$, and the hitting sets $H_t(G)$ are Borel. Indeed, this follows inductively from the piecewise-affine Borel description of $T$: at each finite time one imposes only the Borel inequality $u_{t-1}+d_{t-1}\le a_m$ and then composes with $T$.

Fix a gap $G$ of length $h=|G|<D$. For $p\in H_t(G)$, the $t$th iterate is legal, $u_t\in G$, and \eqref{eq:active-d-bound} gives
\begin{equation}\label{eq:strip-containment}
 T^t(H_t(G))\subseteq G\times[10D,15D).
\end{equation}
By Proposition \ref{prop:dynamics},
\begin{equation}\label{eq:hitting-area-preserved}
 \lambdaTwo(T^t(H_t(G)))=\lambdaTwo(H_t(G)).
\end{equation}

We next show that the image sets in \eqref{eq:strip-containment} are pairwise disjoint. Along every legal orbit starting in $\cR_0$, the difference coordinate never decreases, because $\rho\ge0$; hence $d_j\ge d_0\ge10D$. If $p=(u_0,d_0)\in\cR_0$ admits $r\ge1$ legal applications of $T$, then
\[
 u_r=u_0+\sum_{j=0}^{r-1}d_j\ge u_0+10rD>y+D.
\]
The first coordinate of every point in $\cR_0$ is at most $y+D$, so no positive legal iterate of a point of $\cR_0$ can return to $\cR_0$.

Now suppose that $1\le t<s\le L$ and that
 $T^t p=T^s p'$
for some $p\in H_t(G)$ and $p'\in H_s(G)$. The rollback identity \eqref{eq:rollback} gives $p=T^{s-t}p'$, which is a positive legal return from $\cR_0$ to $\cR_0$, a contradiction. Thus the sets $T^t(H_t(G))$ are pairwise disjoint. They all lie in the strip in \eqref{eq:strip-containment}, whose area is $5Dh$. Using \eqref{eq:hitting-area-preserved}, we obtain
\[
 \sum_{t=1}^L\lambdaTwo(H_t(G))
 =\lambdaTwo\!\left(\bigsqcup_{t=1}^LT^t(H_t(G))\right)
 \le5Dh.
\]
\end{proof}

For a set $E$, let $\mathbf 1_E$ denote its indicator function; that is,
\[
\mathbf 1_E(p)=
\begin{cases}
1,& p\in E,\\
0,& p\notin E.
\end{cases}
\]
We can now convert the packing estimate into a deterministic local wave.

\begin{theorem}\label{thm:local}
Let $D\ge2$, $L\ge1$, $y\in A$, and $J=[y,y+20DL]\subseteq(-\infty,a_m]$. Suppose that every gap intersecting $J$ has length less than $D$ and that
\begin{equation}\label{eq:local-energy-assumption}
 \cE_D^{<}(J)\le\varepsilon D=\frac{D}{40}.
\end{equation}
Then there are points
\[
 y=x_{-1}<x_0<x_1<\cdots<x_L,
 \qquad x_t\in A,
\]
forming an ascending wave and satisfying
\begin{equation}\label{eq:local-difference-bounds}
 \frac D2\le x_0-y<2D,
 \qquad
 9D\le x_{t+1}-x_t\le14D\quad(0\le t<L).
\end{equation}
In particular, the local wave has $L+2$ points, its first difference $x_0-y$ is called the \emph{bridge}, and every later difference is between $9D$ and $14D$.
\end{theorem}

\begin{proof}
Starting from any $p\in\cR_0$, follow the orbit while the accumulated overshoot is less than $D$, and stop either when it first reaches or exceeds $D$ or after $L$ steps, whichever occurs first. The active-window estimate proved inside Lemma \ref{lem:single-gap} shows that every attempted active step is legal and lies in $J$. Thus every initial point has exactly one of the following two outcomes:
\begin{enumerate}[label=(\roman*),leftmargin=9mm]
\item the threshold $D$ is crossed at some active time $\tau\le L$;
\item the orbit admits $L$ legal applications of $T$ and $S_L<D$.
\end{enumerate}
Define the stopped cost
\[
 C(p)=
 \begin{cases}
 D,&\text{in case \textup{(i)}},\\
 S_L(p),&\text{in case \textup{(ii)}}.
 \end{cases}
\]
All stopped sets and the function $C$ are Borel, because they are determined by finitely many Borel inequalities involving $S_t$.

At every active time $t$, either $u_t\in A$, in which case $\rho(u_t)=0$, or $u_t$ lies in a unique gap $G$ intersecting $J$, in which case $p\in H_t(G)$ and $\rho(u_t)<|G|$. If the threshold is crossed at time $\tau$, the crossing hit is included because activity is tested using the pre-hit sum $S_{\tau-1}<D$. Therefore, pointwise on $\cR_0$,
\begin{equation}\label{eq:pointwise-charge}
 C(p)
 \le\sum_{t=1}^L\sum_{G\cap J\ne\emptyset}|G|\,\1_{H_t(G)}(p).
\end{equation}
Indeed, in case \textup{(ii)} the left-hand side equals the sum of all active overshoots, while in case \textup{(i)} the active overshoots through time $\tau$ sum to $S_\tau\ge D$.

Integrating \eqref{eq:pointwise-charge}, applying Lemma \ref{lem:single-gap} to each gap, and then using \eqref{eq:short-gap-energy}, we obtain
\begin{align*}
 \int_{\cR_0}C(p)\,dp
 &\le\sum_{G\cap J\ne\emptyset}|G|
\sum_{t=1}^L\lambdaTwo(H_t(G))\notag\le5D\sum_{G\cap J\ne\emptyset}|G|^2=5D^2\cE_D^{<}(J).
\end{align*}
By \eqref{eq:box-area} and \eqref{eq:local-energy-assumption}, the average stopped cost is at most
\[
 \frac{5D^2\cE_D^{<}(J)}{D^2/2}
 \le10\varepsilon D
 =\frac D4.
\]
If every initial point crossed the threshold, then $C(p)=D$ throughout $\cR_0$, contradicting this average bound. Hence there is a point $p\in\cR_0$ that admits $L$ legal applications of $T$ and satisfies $S_L(p)<D$.

Fix such a point and write $T^t p=(u_t,d_t)$. The active-window estimate gives $u_t\in J$ and $\rho(u_t)<D$ for $1\le t\le L$; the same bound holds at time $0$ because $u_0\in[y+D/2,y+D]\subseteq J$. Since $S_t\le S_L<D$, \eqref{eq:d-sum} gives
\begin{equation}\label{eq:good-orbit-d}
 10D\le d_t<13D\qquad(0\le t\le L).
\end{equation}
In particular, $u_{t+1}-u_t=d_t>0$, so the shadow positions are strictly increasing.

Set $x_t=R(u_t)$ for $0\le t\le L$. For $0\le t<L$,
\begin{align*}
 x_{t+1}-x_t
 &=d_t+\rho(u_{t+1})-\rho(u_t)\\
 &\ge10D-D=9D,
\end{align*}
and, using \eqref{eq:good-orbit-d},
 $x_{t+1}-x_t<13D+D=14D$.
Thus the projected points are strictly increasing. Moreover, the recurrence \eqref{eq:shadow-recurrence} gives $d_t-d_{t-1}=2\rho(u_t)$ for $1\le t<L$, so Lemma \ref{lem:projection} shows that $x_0,\ldots,x_L$ is an ascending wave.

Finally,
 $x_0-y=(u_0-y)+\rho(u_0)$,
where $u_0-y\in[D/2,D]$ and $0\le\rho(u_0)<D$. Hence $D/2\le x_0-y<2D$. Since $2D\le9D\le x_1-x_0$, adjoining the initial point $y$ preserves the ascending-wave property. This proves \eqref{eq:local-difference-bounds} and the theorem.
\end{proof}

\begin{remark}\label{rem:square-energy}
The square in \eqref{eq:short-gap-energy} comes from two separate factors. A hit on a gap $G$ has overshoot at most $|G|$, while Lemma \ref{lem:single-gap} bounds the total phase-space area of all active hits on $G$ by $5D|G|$. Their product is $O(D|G|^2)$; division by the initial area $D^2/2$ gives the averaged contribution $O(|G|^2/D)$.
\end{remark}

\section{Global scale selection and multiscale splicing}\label{sec:global}

We now return to a set $A\subseteq[n]$ with $|A|\ge n/2$. Write
\[
 A=\{a_1<a_2<\cdots<a_m\}.
\]
For $1\le r<m$, let
$G_r=(a_r,a_{r+1})$
be the $r$th gap of $A$, and define its length by
$|G_r|=a_{r+1}-a_r$.
Let
$\mathcal G(A)=\{G_r:1\le r<m\}$
denote the set of all gaps of $A$. Then
\begin{equation}\label{eq:total-gap-length}
\sum_{G\in\mathcal G(A)}|G|
=\sum_{r=1}^{m-1}(a_{r+1}-a_r)
=a_m-a_1
\le n-1<n.
\end{equation}

Fix the absolute constants
\begin{equation}\label{eq:global-constants}
 q=7,
 \qquad Q=2^q=128,
 \qquad C_{\mathrm{win}}=200,
 \qquad
 c_*=\frac{\varepsilon}{2048C_{\mathrm{win}}}.
\end{equation}
For all sufficiently large $n$, define
\begin{equation*}
 \begin{gathered}
 L=\lfloor c_*\log n\rfloor,
 \qquad
 N=\left\lfloor\log_2\!\left(\frac{n}{64C_{\mathrm{win}}L}\right)\right\rfloor,\\
 D_i=2^i,
 \qquad
 w_i=C_{\mathrm{win}}D_iL,
 \qquad
 T_i=w_i+D_i
 \qquad(1\le i\le N).
 \end{gathered}
\end{equation*}
Here $D_i$ is the working scale, $w_i$ is the main window length, and the additional $D_i$ in $T_i$ is a projection buffer.

We first record the elementary
bounds on these parameters that will be used repeatedly.
\begin{lemma}\label{lem:parameter-bounds}
For all sufficiently large $n$, the parameters above satisfy
\begin{equation}\label{eq:constant-estimates}
 L\ge1,
 \quad N\ge\frac14\log n,
 \quad w_i\le\frac n{64},
 \quad T_i+1\le2w_i
 \qquad(1\le i\le N),
\end{equation}
and
\begin{equation}\label{eq:loss-estimates}
 16C_{\mathrm{win}}\frac LN\le\frac1{16},
 \qquad
 \frac{16C_{\mathrm{win}}}{\varepsilon}\frac LN\le\frac1{16}.
\end{equation}
\end{lemma}

\begin{proof}
Since $L=\lfloor c_*\log n\rfloor$, we have $L\ge1$ for all sufficiently large $n$ and $L\le c_*\log n$. Moreover,
\[
 N
 \ge \frac{\log n}{\log 2}
    -\frac{\log(64C_{\mathrm{win}}L)}{\log 2}-1
 =\frac{\log n}{\log 2}-O(\log\log n).
\]
Hence $N\ge\frac14\log n$ for all sufficiently large $n$.

For $1\le i\le N$, the definition of $N$ gives
\[
 D_i=2^i\le2^N\le\frac{n}{64C_{\mathrm{win}}L},
\]
and therefore $w_i=C_{\mathrm{win}}D_iL\le n/64$. Since $D_i\ge2$, $L\ge1$, and $C_{\mathrm{win}}=200$, we also have $D_i+1\le w_i$, whence
\[
 T_i+1=w_i+D_i+1\le2w_i.
\]
Finally, $L/N\le4c_*$ by the first two estimates. Thus
\[
 16C_{\mathrm{win}}\frac LN
 \le64C_{\mathrm{win}}c_*
 =\frac{\varepsilon}{32}
 \le\frac1{16},
\]
and
\[
 \frac{16C_{\mathrm{win}}}{\varepsilon}\frac LN
 \le\frac{64C_{\mathrm{win}}c_*}{\varepsilon}
 =\frac1{32}
 \le\frac1{16}.
\]
\end{proof}

For $k\ge0$, let $\cG_k$ be the family of gaps $G$ with $2^k\le|G|<2^{k+1}$,
and set
\[
 u_k=\sum_{G\in\cG_k}|G|.
\]
The classes $\cG_k$ are the \emph{dyadic gap classes}. Because all gap lengths are positive integers, they contain every gap, and \eqref{eq:total-gap-length} gives $\sum_{k\ge0}u_k\le n$.

At scale $D_i$, define the \emph{large-gap pressure}
$$B_i=\sum_{j\ge0}2^{-j}u_{i+j}$$
and the \emph{global short-gap square energy}
 $$W_i=\sum_{G:|G|<D_i}\frac{|G|^2}{D_i}.$$
The pressure controls how many windows can meet gaps of length at least $D_i$, whereas $W_i$ controls the average local square energy contributed by shorter gaps.

A point $x\in A$ is called \emph{admissible at scale $i$} if all three conditions below hold:
\begin{enumerate}[label=(\roman*),leftmargin=9mm]
\item $x+T_i\le a_m$; this is the \emph{right-safety condition};
\item the interval $[x,x+T_i]$ meets no gap of length at least $D_i$;
\item $\cE_{D_i}^{<}([x,x+T_i])\le\varepsilon D_i$.
\end{enumerate}

Then, we find a common starting point that
is admissible at many scales.
\begin{proposition}\label{prop:common-point}
For all sufficiently large $n$, there are a point $x^*\in A$ and a set $I\subseteq[N]$ such that $x^*$ is admissible at every scale $i\in I$ and
\begin{equation}\label{eq:many-scales}
 |I|\ge\frac N8\ge\frac1{32}\log n.
\end{equation}
\end{proposition}

\begin{proof}
We first show that many scales have controlled global statistics. A fixed class $u_k$ occurs in $B_i$ only for $i\le k$, with coefficient $2^{-(k-i)}$. Hence its total coefficient in $\sum_{i=1}^NB_i$ is at most $\sum_{r\ge0}2^{-r}=2$, and therefore
\begin{equation}\label{eq:sum-pressure}
 \sum_{i=1}^NB_i\le2\sum_{k\ge0}u_k\le2n.
\end{equation}
For a fixed gap of length $h$, let $i_0$ be the first index with $D_{i_0}>h$, if such an index exists. Then
\[
 \sum_{\substack{1\le i\le N\\D_i>h}}\frac{h^2}{D_i}
 \le\frac{h^2}{D_{i_0}}\sum_{r\ge0}2^{-r}
 <2h.
\]
Summing over all gaps gives
\begin{equation}\label{eq:sum-energy}
 \sum_{i=1}^NW_i\le2\sum_G|G|\le2n.
\end{equation}
It follows from \eqref{eq:sum-pressure} and \eqref{eq:sum-energy} that at least $N/2$ indices satisfy both
$$
 B_i\le\frac{8n}{N},
 \qquad
 W_i\le\frac{8n}{N}.
$$
Call such an index a \emph{good scale}.

Fix a good scale $i$. We estimate the number of points of $A$ that fail one of the three admissibility conditions. First, at most $T_i+1\le2w_i\le n/32$ points fail right safety.

Next consider a gap $G=(a,b)$. If $x\in A$ and $[x,x+T_i]$ intersects $G$, then $x\notin G$, $x<b$, and $x+T_i>a$. Hence
$x\in(a-T_i,a]\cap\mathbb Z$,
so at most $T_i+1\le2w_i$ possible starting points touch $G$. A gap in $\cG_{i+j}$ has length at least $D_i2^j$, so the number of such gaps is at most $u_{i+j}/(D_i2^j)$. By the union bound, the number of points whose window meets a gap of length at least $D_i$ is at most
\begin{align*}
 2w_i\sum_{j\ge0}\frac{u_{i+j}}{D_i2^j}
 &=2C_{\mathrm{win}}LB_i\le16C_{\mathrm{win}}L\frac nN
 \le\frac n{16},
\end{align*}
where the last inequality is \eqref{eq:loss-estimates}.

For the energy condition, count a short gap once for every starting point whose window touches it. The same touching estimate gives
$$
 \sum_{x\in A}\cE_{D_i}^{<}([x,x+T_i])\le2w_iW_i.
$$
By Markov's inequality, the number of $x\in A$ for which this energy exceeds $\varepsilon D_i$ is at most
\begin{align*}
 \frac{2w_iW_i}{\varepsilon D_i}
 &\le\frac{16C_{\mathrm{win}}}{\varepsilon}L\frac nN
 \le\frac n{16}.
\end{align*}
Since $|A|\ge n/2$, after removing the points counted above at least
\[
 \frac n2-\frac n{32}-\frac n{16}-\frac n{16}
 =\frac{11n}{32}>\frac n4
\]
points remain admissible at the fixed good scale.

There are at least $N/2$ good scales, each with at least $n/4$ admissible points. Hence the number of admissible pairs $(x,i)$ is at least $nN/8$. Averaging over the $n$ integers in $[n]$, some point $x^*\in A$ is admissible at at least $N/8$ scales. Let $I$ be the set of those scales. Then $|I|\ge N/8$, and \eqref{eq:constant-estimates} yields \eqref{eq:many-scales}.
\end{proof}

To combine local waves, we use the following terminology. Suppose an existing ascending wave ends at a point $y$, and a new local wave begins with the same point $y$. Their \emph{splice} is the concatenation obtained by writing the common endpoint only once. The first new difference is the bridge of the local wave.

We now convert the multiscale admissibility of $x^*$ into a single long
ascending wave by selecting well-separated scales and splicing the
corresponding local waves.

\begin{proposition}\label{prop:splicing}
Let $x^*$ and $I$ be as in Proposition \ref{prop:common-point}. Then $A$
contains an ascending wave of length at least
\[
\frac{|I|L}{q}.
\]
\end{proposition}

\begin{proof}
Recall that
$q=7,
Q=2^q=128$.
Partition $I$ according to the residue classes modulo $q$. At least one
residue class, denoted by $I'$, satisfies
\[
|I'|\ge \frac{|I|}{q}.
\]
Write
$I'=\{i_1<i_2<\cdots<i_s\}$.
Since all indices in $I'$ are congruent modulo $q$, we have
$i_{r+1}-i_r\ge q$
for every $1\le r<s$. As $D_i=2^i$, it follows that
$D_{i_{r+1}}\ge 2^qD_{i_r}=QD_{i_r}$.

We shall apply Theorem \ref{thm:local} successively at the scales
$D_{i_1},D_{i_2},\ldots,D_{i_s}$. Suppose that Theorem \ref{thm:local} is applied at a point $y\in A$ and at
scale $D$. It produces points
\[
y=z_{-1}<z_0<z_1<\cdots<z_L
\]
such that
\[
\frac D2\le z_0-y<2D
\]
and
\[
9D\le z_{t+1}-z_t\le14D
\qquad (0\le t<L).
\]
Consequently,
\begin{align*}
z_L-y
&=(z_0-y)+\sum_{t=0}^{L-1}(z_{t+1}-z_t)\\
&<2D+14LD\\
&\le16LD,
\end{align*}
where the last inequality uses $L\ge1$. Moreover, the last difference
of this local segment satisfies
$z_L-z_{L-1}\le14D$.

We now construct the required wave inductively. Before scale
$D_{i_r}$ is used, let $y_r$ denote the current endpoint. For $r=1$,
set
$y_1=x^*$.
For $r>1$, the point $y_r$ is the final point of the local segment
constructed at scale $D_{i_{r-1}}$.

The total displacement from $x^*$ caused by all scales preceding
$D_{i_r}$ is at most
$16L\sum_{a<r}D_{i_a}$.
The geometric separation of the selected scales gives
\[
D_{i_a}\le Q^{-(r-a)}D_{i_r}
\qquad (a<r).
\]
Therefore,
\begin{align*}
y_r-x^*
&\le16L\sum_{a<r}D_{i_a}
\le16LD_{i_r}\sum_{j=1}^{\infty}Q^{-j}
=\frac{16}{Q-1}LD_{i_r}\le\frac{32}{Q}LD_{i_r}\le LD_{i_r}
\le\frac{w_{i_r}}2.
\end{align*}
Thus,
$x^*\le y_r\le x^*+\frac{w_{i_r}}2$.
Consider the local window
$J_r=[y_r,y_r+20D_{i_r}L]$.

Since
\[
20D_{i_r}L=\frac{w_{i_r}}{10},
\]
we obtain
\[
y_r+20D_{i_r}L
\le x^*+\frac{w_{i_r}}2+\frac{w_{i_r}}{10}
<x^*+w_{i_r}.
\]
Hence
\[
J_r
\subseteq[x^*,x^*+w_{i_r}]
\subseteq[x^*,x^*+T_{i_r}]
\subseteq(-\infty,a_m].
\]

Because $x^*$ is admissible at scale $i_r$, the interval
$[x^*,x^*+T_{i_r}]$ meets no gap of length at least $D_{i_r}$.
Therefore every gap meeting $J_r$ has length less than $D_{i_r}$.
Moreover, every gap meeting $J_r$ also meets
$[x^*,x^*+T_{i_r}]$, and hence
\[
\cE_{D_{i_r}}^{<}(J_r)
\le
\cE_{D_{i_r}}^{<}\bigl([x^*,x^*+T_{i_r}]\bigr)
\le\varepsilon D_{i_r}.
\]

We next verify that the right projections of all tentative positions remain inside
the buffered reservoir $[x^*,x^*+T_{i_r}]$. Let
$u\in[x^*,x^*+w_{i_r}]$.
If $u\in A$, then $R(u)=u$. Otherwise, $u$ lies in a gap meeting
$[x^*,x^*+T_{i_r}]$. This gap has length less than $D_{i_r}$, and
therefore
\[
R(u)=u+\rho(u)
<u+D_{i_r}
\le x^*+w_{i_r}+D_{i_r}
=x^*+T_{i_r}
\le a_m.
\]
Thus all the hypotheses of Theorem \ref{thm:local} hold at the point $y_r$
and scale $D_{i_r}$.

Let the local segment supplied by Theorem \ref{thm:local} be
\[
y_r=z_{r,-1}<z_{r,0}<z_{r,1}<\cdots<z_{r,L}.
\]
Set
$\beta_r=z_{r,0}-y_r$
for its bridge difference, and
\[
\delta_{r,t}=z_{r,t+1}-z_{r,t}
\qquad (0\le t<L)
\]
for its internal differences. Then
\[
\frac{D_{i_r}}2\le\beta_r<2D_{i_r}
\]
and
\[
9D_{i_r}\le\delta_{r,t}\le14D_{i_r}
\qquad (0\le t<L).
\]

For $r=1$, the existing sequence consists only of the point $x^*$, so
the local segment can be appended without any compatibility condition.
Now suppose that $r>1$. Let $\Delta_r$ be the last difference of the
wave constructed before scale $D_{i_r}$ is used. This is the last
difference of the local segment constructed at scale
$D_{i_{r-1}}$, and hence
$\Delta_r\le14D_{i_{r-1}}$.
Since
$D_{i_r}\ge QD_{i_{r-1}}$,
we have
\[
\Delta_r
\le\frac{14}{Q}D_{i_r}
<\frac{D_{i_r}}2.
\]
It follows that
\[
\Delta_r
<\frac{D_{i_r}}2
\le\beta_r
<2D_{i_r}
<9D_{i_r}
\le\delta_{r,0}.
\]
Thus the last difference of the previously constructed wave is at most
the new bridge difference, and the bridge difference is at most the
first internal difference of the new local segment. Since the internal
differences $\delta_{r,0},\delta_{r,1},\ldots,\delta_{r,L-1}$
are nondecreasing, appending
$z_{r,0},z_{r,1},\ldots,z_{r,L}$
preserves the ascending-wave property.

After this splice, the new endpoint is $z_{r,L}$, the displacement
created at scale $D_{i_r}$ is at most $16LD_{i_r}$, and the new last
difference is at most $14D_{i_r}$. These are precisely the bounds used
above to continue the induction at scale $D_{i_{r+1}}$.

Each selected scale contributes the $L+1$ new points
$z_{r,0},z_{r,1},\ldots,z_{r,L}$.
In particular, it contributes at least $L$ new points. Therefore the
final ascending wave has length at least
$sL$.
Since
\[
s=|I'|\ge\frac{|I|}{q},
\]
its length is at least
\[
sL\ge\frac{|I|L}{q}.
\]
This proves the proposition.
\end{proof}
\begin{proof}[Proof of Theorem \ref{thm:main}]
For all sufficiently large $n$, apply Propositions \ref{prop:common-point} and \ref{prop:splicing}. By \eqref{eq:many-scales},
\[
 |I|\ge\frac1{32}\log n.
\]
Also, once $c_*\log n\ge2$,
\[
 L=\lfloor c_*\log n\rfloor\ge\frac{c_*}{2}\log n.
\]
Therefore, the wave constructed in Proposition \ref{prop:splicing} has length at least
\[
 \frac{|I|L}{q}
 \ge\frac{c_*}{64q}(\log n)^2.
\]
This proves the theorem for all $n\ge n_0$, for some absolute $n_0$.

For $1\le n<n_0$, every set $A\subseteq[n]$ with $|A|\ge n/2$ is nonempty and hence contains a one-point ascending wave. Set
\[
 c=\min\left\{\frac{c_*}{64q},
 \min_{2\le n<n_0}\frac1{(\log n)^2}\right\},
\]
where the second minimum is omitted if $n_0\le2$. Then $c>0$, and the asserted lower bound holds for every $n\in\N$; for $n=1$ the right-hand side is $0$.
\end{proof}

\begin{corollary}\label{cor:asymptotic}
The dense-set ascending-wave function satisfies
\[
 g(n)=\Theta((\log n)^2).
\]
\end{corollary}

\begin{proof}
The lower bound is Theorem \ref{thm:main}. The matching upper bound is due to Alon and Spencer~\cite{AS}.
\end{proof}


\begin{thebibliography}{99}

\bibitem{AS}
N. Alon and J. H. Spencer,
Ascending waves,
\emph{J. Combin. Theory Ser. A} 52(2) (1989), 275--287.

\bibitem{BBCY}
A. Bialostocki, G. Bialostocki, Y. Caro and R. Yuster,
Zero-sum ascending waves,
\emph{J. Combin. Math. Combin. Comput.} 32 (2000), 103--114.

\bibitem{BERJ}
B. Bollob\'as, P. Erd\H{o}s and G. Jin,
Strictly ascending pairs and waves,
in Y. Alavi and A. Schwenk (eds.), \emph{Graph Theory, Combinatorics, and Algorithms}, Vols.~1--2 (Kalamazoo, MI, 1992), Wiley-Interscience, New York, 1995, 83--95.

\bibitem{BEF}
T. C. Brown, P. Erd\H{o}s and A. R. Freedman,
Quasi-progressions and descending waves,
\emph{J. Combin. Theory Ser. A} 53(1) (1990), 81--95.

\bibitem{Cong}
K. Cong,
On integer sets excluding permutation pattern waves,
arXiv:2308.15695 [math.CO], 2023.

\bibitem{LR}
B. M. Landman and A. Robertson,
Monochromatic strictly ascending waves and permutation pattern waves,
\emph{Adv. Appl. Math.} 146 (2023), 102501.

\bibitem{LSR}
T. LeSaulnier and A. Robertson,
On monochromatic ascending waves,
\emph{Integers} 7(2) (2007), A23.

\bibitem{RCDX}
A. Robertson, C. Cremin, W. Daniel and Q. Xiang,
Intermingled ascending wave $m$-sets,
\emph{Discrete Math.} 339(2) (2016), 560--563.
\end{thebibliography}
\end{document}